\documentclass[a4paper,notitlepage,twoside,leqno,10pt]{amsart}

\usepackage{bbm,pifont,latexsym}

\usepackage{dcolumn,indentfirst}
\usepackage[bookmarksopen=true,linktocpage=true,pdfstartview={XYZ null null 1.25}]{hyperref}
\usepackage{amsmath,amssymb,amscd,amsthm,amsfonts,mathrsfs}
\usepackage{color,graphicx,xcolor,graphics,subfigure,extarrows,caption2}
\usepackage{titletoc}

\newtheorem{thm}{Theorem}[section]
\newtheorem{lema}[thm]{Lemma}

\newtheorem{prop}[thm]{Proposition}
\newtheorem*{mainthm}{Main Theorem}

\theoremstyle{definition}

\theoremstyle{remark}
\newtheorem*{rmk}{Remark}

\newcommand{\EC}{\widehat{\mathbb{C}}}
\newcommand{\C}{\mathbb{C}}
\newcommand{\D}{\mathbb{D}}

\newcommand{\N}{\mathbb{N}}

\newcommand{\R}{\mathbb{R}}

\newcommand{\Z}{\mathbb{Z}}

\newcommand{\MA}{\mathcal{A}}

\newcommand{\MN}{\mathcal{N}}

\newcommand{\ii}{\textup{i}}
\newcommand{\re}{\textup{Re}}
\newcommand{\im}{\textup{Im}}

\newcommand{\diam}{\textup{diam}}

\begin{document}

\title[Cantor Julia sets with Hausdorff dimension two]{Cantor Julia sets with Hausdorff dimension two}

\author[FEI YANG]{FEI YANG}
\address{Department of Mathematics, Nanjing University, Nanjing 210093, P. R. China}
\email{yangfei@nju.edu.cn}

\begin{abstract}
We prove the existence of Cantor Julia sets with Hausdorff dimension two. In particular, such examples can be found in cubic polynomials. The proof is based on the characterization of the parameter spaces and dynamical planes of cubic polynomials by Branner-Hubbard, and the parabolic bifurcation theory developed by Shishikura in 1990s.
\end{abstract}



\date{\today}



\maketitle


\section{Introduction}\label{introduction}

A nonempty subset $X$ in $\C$ or $\EC$ is called a \textit{Cantor set} if $X$ is compact, perfect and totally disconnected. As the chaotic sets of holomorphic dynamics, the Julia sets can exhibit rich topological structures. In particular, some of them are Cantor sets. It was shown by Fatou and Julia in 1920s that if all the critical points of a polynomial escape to the infinity under iteration, then the Julia set of this polynomial is a Cantor set (\cite{Fat20}, \cite{Jul18}). There are also some criterions to obtain the Cantor Julia sets for rational maps (see \cite{Mil93} and \cite{Yin92}). The first examples of transcendental meromorphic functions with Cantor Julia sets was produced by Devaney and Keen in \cite{DK88} (see also \cite{Kos10} for different examples).

It was known that for any $s\in[0,2]$, there is a Cantor set in $\C$ whose Hausdorff dimension is exactly $s$ (see \cite{Kee86} for example). However, for the Cantor \textit{Julia} sets, this result does not hold. Indeed, the Hausdorff dimension of the Julia set of any non-constant meromorphic function (not a M\"{o}bius map) is positive (see \cite{Gar78}, \cite{Sta94} and the references therein). According to \cite{BZ96} and \cite{Shi98}, the Hausdorff dimension of the Cantor Julia sets of quadratic polynomials can take any value in $(0,2)$. But the existence of the Cantor Julia sets with full Hausdorff dimension was unknown. Our main purpose in this article is to fill this gap.

\begin{mainthm}\label{thm-dim-2}
There exist cubic polynomials whose Julia sets are Cantor sets having Hausdorff dimension two.
\end{mainthm}

This is the simplest kind of examples one can find such that the Cantor Julia sets have full Hausdorff dimension. Indeed, if a quadratic rational map has a Cantor Julia set then this Julia set cannot contain any critical points (\cite{Mil93} and \cite{Yin92}) and always have Hausdorff dimension strictly less than two.

It was known from \cite[Theorem 5.9]{BH92} that the Cantor Julia sets of cubic polynomials have zero Lebesgue area. Hence we have found some Julia sets of polynomials such that they have zero area but have full Hausdorff dimension. Such kind of examples was known very early in the family of exponential maps \cite{McM87}, and was recently constructed for some Feigenbaum polynomials (\cite{AL08} and \cite{AL15}).

In order to find the Cantor Julia sets with full dimension, we should first exclude the rational maps in the ``shift locus" since the maps there are hyperbolic and the corresponding Julia sets are Cantor sets with Hausdorff dimension strictly less than two. Hence the natural candidates in the Main Theorem are the maps on the boundary of the cubic shift locus.

Let $P_{a,b}$ be the \textit{monic} and \textit{centered} cubic polynomial which is defined by
\begin{equation*}
P_{a,b}(z)=z^3-3a^2z+b, \text{ where } a,b\in\C.
\end{equation*}
According to B\"{o}ttcher's theorem, there exists a unique conformal isomorphism $\varphi_{a,b}$ defined in a neighborhood of $\infty$ which is tangent to the identity at $\infty$ such that $\varphi_{a,b}$ conjugates $P_{a,b}$ to $\zeta\mapsto\zeta^3$. The \textit{\textit{potential function}}
\begin{equation*}
h_{a,b}(z)=\lim_{k\to\infty}\frac{1}{3^k}\log_+|P_{a,b}^{\circ k}(z)|
\end{equation*}
is defined in $\C$ and coincides with $\log|\varphi_{a,b}(z)|$ in a neighborhood of $\infty$. Moreover, $h_{a,b}(z)>0$ if and only if $z$ escapes to $\infty$ under iteration. For $z_0\in\C$, we denote
\begin{equation*}
U_{a,b}(z_0)=\{z\in\C:h_{a,b}(z)>h_{a,b}(z_0)\}.
\end{equation*}

The map $P_{a,b}$ has two critical points $\pm a$ whose forward orbits determine the dynamics essentially. In order to obtain the Cantor Julia set, we need to assume that one critical orbit of $P_{a,b}$ is unbounded. For convenience we assume that $h_{a,b}(a)>0$ and $h_{a,b}(-a)< h_{a,b}(a)$. It is well known that the map $\varphi_{a,b}$ can be extended to a conformal isomorphism $\varphi_{a,b}:U_{a,b}(a)\to\C\setminus\overline{\D}_r$ for some $r>1$, where $\D_r:=\{z\in\C:|z|<r\}$. In particular, $\lim_{z\to -2a}\varphi_{a,b}(z)$ is well-defined, where $P_{a,b}(-2a)=P_{a,b}(a)$ and $-2a\in\partial U_{a,b}(a)$. For each fixed $\zeta\in\C\setminus\overline{\D}$, we define
\begin{equation}\label{equ:L-plus}
\mathscr{L}^+(\zeta):=\{(a,b)\in\C^2: \lim_{z\to -2a}\varphi_{a,b}(z)=\zeta \text{ and }h_{a,b}(-a)<h_{a,b}(a)\}.
\end{equation}
It was shown by Branner-Hubbard that for each $\zeta\in\C\setminus\overline{\D}$, $\mathscr{L}^+(\zeta)$ is homeomorphic to a disk, which is a leaf of a clover \cite[Corollary 13.3]{BH88}.

\begin{figure}[!htpb]
  \setlength{\unitlength}{1mm}
  \centering
  \includegraphics[width=0.32\textwidth]{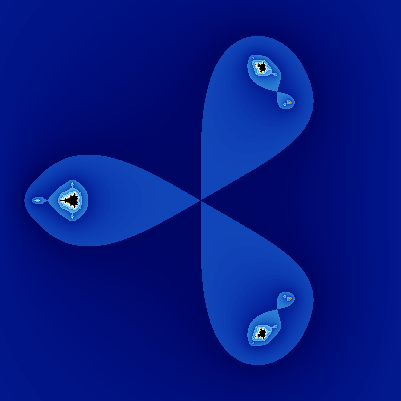}\hskip0.2cm
  \includegraphics[width=0.32\textwidth]{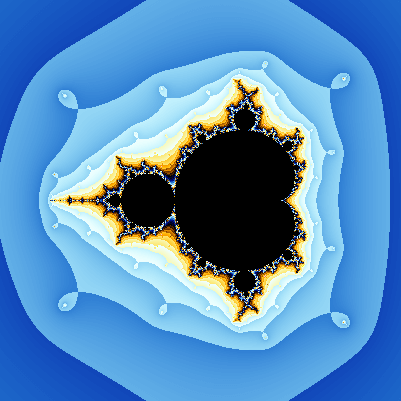}\hskip0.2cm
  \includegraphics[width=0.32\textwidth]{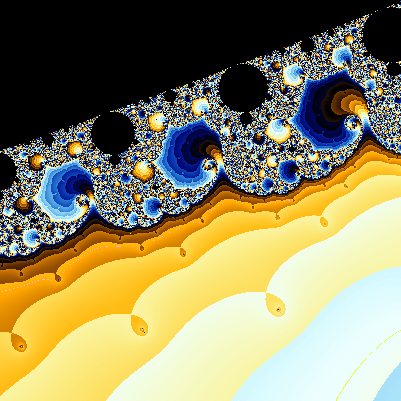}\vskip0.2cm
  \includegraphics[width=0.32\textwidth]{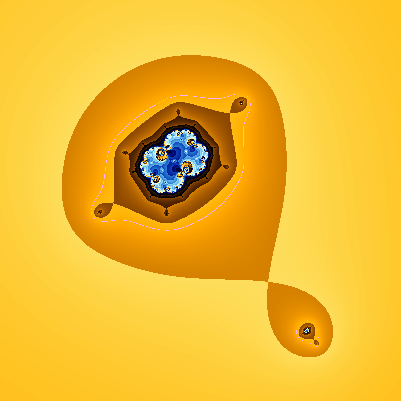}\hskip0.2cm
  \includegraphics[width=0.32\textwidth]{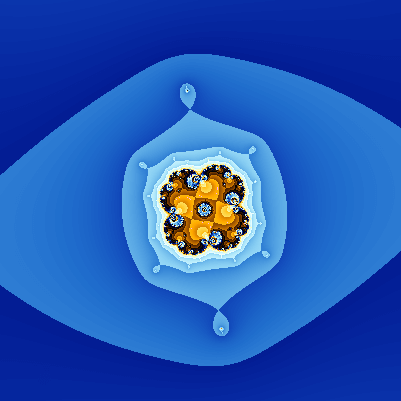}\hskip0.2cm
  \includegraphics[width=0.32\textwidth]{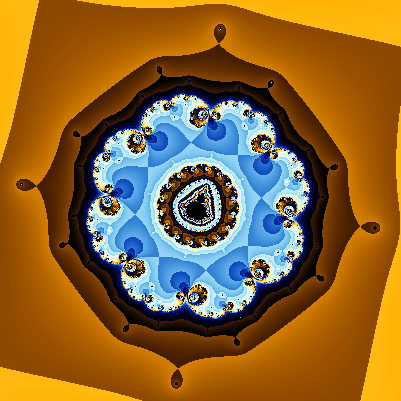}
  \caption{The space $\mathscr{L}^+(\zeta)$ for some $\zeta>1$. The set $\mathscr{B}^+(\zeta)\subset\mathscr{L}^+(\zeta)$ has been drawn and zoomed in several times. The copies of the Mandelbrot set and some decorations of the point components of $\mathscr{B}^+(\zeta)$ can be seen clearly.}
  \label{Fig_M-Cantor}
\end{figure}

To get the potential full Hausdorff dimension of the Cantor Julia set, $P_{a,b}$ needs to have a bounded critical orbit. Define
\begin{equation}\label{equ:B-plus}
\mathscr{B}^+(\zeta):=\{(a,b)\in\mathscr{L}^+(\zeta): h_{a,b}(-a)=0\}.
\end{equation}
Each component of $\mathscr{B}^+(\zeta)$ is either a point or a copy of the Mandelbrot set (Theorem \ref{thm:pt-or-copy}). Moreover, all these Mandelbrot copies are accumulated by the point components of $\mathscr{B}^+(\zeta)$ (Proposition \ref{prop:accum}). See Figure \ref{Fig_M-Cantor}.

\vskip0.2cm
For the proof of the Main Theorem, we fix $\mathscr{L}^+(\zeta)$ for some $\zeta\in\C\setminus\overline{\D}$. Then $\mathscr{B}^+(\zeta)$ consists of countably many components $(M^n)_{n\geq 1}$ which are the copies of the Mandelbrot sets and uncountably many components which are points (see Theorem \ref{thm:pt-or-copy}). We begin with a parabolic polynomial $P_{a_1,b_1}$ in $M^1$. According to Shishikura (see Theorem \ref{thm-shishi}), one can find a punctured ``neighborhood" $U_1$ of $(a_1,b_1)$ in $\mathscr{L}^+(\zeta)$ such that the Julia sets of all $P_{a,b}$ in $U_1$ have Hausdorff dimension at least $1$. Next we find a parabolic polynomial $P_{a_2,b_2}$ (which is not in $M^1$) in $U_1$ and find a punctured ``neighborhood" $U_2\Subset U_1$ of $(a_2,b_2)$ in $\mathscr{L}^+(\zeta)$ such that the Julia sets of all $P_{a,b}$ in $U_2$ have Hausdorff dimension at least $3/2$. Inductively, one can choose a sequence of parabolic parameters $(a_n,b_n)_{n\geq 1}\subset\partial\mathscr{B}^+(\zeta)$ and a sequence of nested ``neighborhoods" $(U_n)_{n\geq 1}$ of $(a_n,b_n)_{n\geq 1}$ in $\mathscr{L}^+(\zeta)$ with $U_n\Subset U_{n-1}$ such that $(a_n,b_n)$ is not in $M^1\cup\cdots\cup M^{n-1}$, and the Julia sets of all $P_{a,b}$ in $U_n$ have Hausdorff dimension at least $2-1/n$. Therefore, we can obtain a limit $(a^*,b^*)$ such that the Julia set of $P_{a^*,b^*}$ has Hausdorff dimension two. By the choice of $(a_n,b_n)$, the limit $(a^*,b^*)$ is not contained in any copies of the Mandelbrot set of $\mathscr{B}^+(\zeta)$ and the Julia set of $P_{a^*,b^*}$ must be a Cantor set.

In fact, by the density of the parabolic parameters on the bifurcation locus of $P_{a,b}$ in $\mathscr{L}^+(\zeta)$, we can show that the parameters such that the Cantor Julia sets have Hausdorff dimension two are dense in $\partial\mathscr{B}^+(\zeta)$ (Theorem \ref{thm-dense}). One can compare this result with Shishikura's theorem on the density of the parameters on the boundary of the Mandelbrot set such that the quadratic Julia sets have full Hausdorff dimension \cite{Shi98}.

Naturally, a further problem is find the Cantor Julia sets with positive area. However, this kind of Cantor Julia sets may not be exist (see \cite{YZ10}). One can refer to \cite{Zha13} for the study of the dimensions on Cantor Julia sets.

\vskip0.2cm
\textit{Organization of the article}. In \S\ref{sec-BH} we characterize the structure of the connected components of $\mathscr{B}^+(\zeta)$ based on the work of Branner-Hubbard. This guarantees that the desired perturbations in the proof of the Main Theorem are feasible. In \S\ref{sec-Shi} we recall Shishikura's result on the Hausdorff dimension and complete the proof of the Main Theorem. Finally in \S\ref{sec-exam} we give a brief argument to generalize the Main Theorem to some higher degree polynomials and some bi-parameter rational maps.

\vskip0.2cm
\textit{Acknowledgements}. I would like to thank Yongcheng Yin for introducing me to these questions. This work is supported by NSFC under grant No.\,11671092.

\section{Slices of the parameter spaces of cubic polynomials}\label{sec-BH}

For a polynomial $P$ with degree at least two, the \textit{filled-in Julia set} $K_P$ is defined as the set of points with bounded orbit under the iteration of $P$. For the cubic polynomial $P_{a,b}$, the filled-in Julia set has the following equivalent definition
\begin{equation*}
K_{a,b}:=\{z\in\C:h_{a,b}(z)=0\},
\end{equation*}
where $h_{a,b}$ is the potential function defined in the introduction. For any $z\in K_{a,b}$, we denote by $K_{a,b}(z)$ the connected component of $K_{a,b}$ containing $z$. A component $K_{a,b}(z)$ is called \textit{periodic} if $P_{a,b}^{\circ k}(K_{a,b}(z))=K_{a,b}(z)$ for some $k\geq 1$.

Recall that $\mathscr{B}^+(\zeta)$ is defined in \eqref{equ:B-plus} for $\zeta\in\C\setminus\overline{\D}$. For $(a,b)\in \mathscr{B}^+(\zeta)$, we denote by $\mathscr{B}_{a,b}^+(\zeta)$ the connected component of $\mathscr{B}^+(\zeta)$ containing $(a,b)$. The following theorem was proved in \cite[Theorems 5.2 and 9.1]{BH92}.

\begin{thm}[{Branner-Hubbard}]\label{thm:pt-or-copy}
For any $(a,b)\in\mathscr{B}^+(\zeta)$ with $\zeta\in\C\setminus\overline{\D}$, there is a following dichotomy
\begin{enumerate}
\item If $K_{a,b}(-a)$ is periodic, then $\mathscr{B}_{a,b}^+(\zeta)$ is homeomorphic to the Mandelbrot set $M$ and $K_{a,b}(-a)$ is a quasiconformal copy of the filled-in Julia set of $z\mapsto z^2+c$ for some $c\in M$;
\item If $K_{a,b}(-a)$ is not periodic, then $\mathscr{B}_{a,b}^+(\zeta)$ is a point and the Julia set of $P_{a,b}$ is a Cantor set.
\end{enumerate}
Moreover, $\mathscr{B}^+(\zeta)$ has countably many of the first kind (non-degenerated) components and uncountably many of the second (degenerated).
\end{thm}

Recall that $\mathscr{L}^+(\zeta)$ is defined in \eqref{equ:L-plus}.
The following result shows that each boundary point in the non-degenerated connected component of $\mathscr{B}^+(\zeta)$ is accumulated by the degenerated connected component of $\mathscr{B}^+(\zeta)$.

\begin{prop}\label{prop:accum}
Let $\mathscr{B}_{a,b}^+(\zeta)$ be a non-degenerated connected component of $\mathscr{B}^+(\zeta)$, where $\zeta\in\C\setminus\overline{\D}$. For any open set $U$ in $\mathscr{L}^+(\zeta)$ which intersects $\partial\mathscr{B}_{a,b}^+(\zeta)$, there exists a degenerated connected component of $\mathscr{B}^+(\zeta)$ contained in $U$.
\end{prop}

\begin{proof}
This proposition can be proved by using the relation of the structures of ``\textit{patterns}" (corresponding to dynamical planes) and ``\textit{parapatterns}" (corresponding to parameter spaces) introduced in \cite{BH92}. However, to avoid the long definitions we prefer to use the idea of perturbation here.

Let $U$ be an open set in $\mathscr{L}^+(\zeta)$ such that $U\cap\partial\mathscr{B}_{a,b}^+(\zeta)\neq\emptyset$ for some $(a,b)\in\mathscr{B}^+(\zeta)$, where $\mathscr{B}_{a,b}^+(\zeta)$ is non-degenerated. Since $\mathscr{B}_{a,b}^+(\zeta)$ is homeomorphic to the Mandelbrot set $M$, by the density of the parabolic parameters on $\partial M$, there exists a point $(a_0,b_0)\in U\cap\partial\mathscr{B}_{a,b}^+(\zeta)$ such that $P_{a_0,b_0}$ has a parabolic periodic point $z_0$. Note that the critical point $-a$ is contained in the parabolic basin of $z_0$.

The Julia set of $P_{a_0,b_0}$ consists of countably many quasiconformal copies of the Julia set of a parabolic polynomial $z\mapsto z^2+c$ and uncountably many point components. In particular, there exists a repelling periodic orbit
\begin{equation*}
X_{a_0,b_0}:=\{z_{a_0,b_0}^1,z_{a_0,b_0}^2,\cdots,z_{a_0,b_0}^k\}
\end{equation*}
with period $k\geq 4$ such that $P_{a_0,b_0}(z_{a_0,b_0}^i)=z_{a_0,b_0}^{i+1}$ $(i\in\Z/k\Z)$ and each $z_{a_0,b_0}^i$ is a Julia component of $P_{a_0,b_0}$.

Note that $\mathscr{L}^+(\zeta)\subset\C^2$ is a Riemann surface (satisfying $\varphi_{a,b}(-2a)=\zeta$ for fixed $|\zeta|>1$) which is isomorphic to the unit disk \cite[Corollary 13.3]{BH88} and
\begin{equation}\label{equ:holo-fami}
P_{a,b}:\mathscr{L}^+(\zeta)\times\C\to\C, \quad ((a,b),z)\mapsto P_{a,b}(z)
\end{equation}
is a holomorphic family of cubic polynomials. There exists an open neighborhood $U'\subset U$ of $(a_0,b_0)$ in $\mathscr{L}^+(\zeta)$ and a repelling periodic orbit $Y_{a,b}$ of $P_{a,b}$ with period $k$ such that
\begin{enumerate}
\item $Y_{a,b}$ moves holomorphically as $(a,b)\in U'$ and $Y_{a_0,b_0}=X_{a_0,b_0}$;
\item For $(a,b)\in U'$, the set $X_{a,b}=\{z_{a,b}^1,z_{a,b}^2,\cdots,z_{a,b}^k\}$ satifies $P_{a,b}(z_{a,b}^i)=z_{a,b}^{i+1}$ $(i\in\Z/k\Z)$; and
\item The critical point $-a$ is not contained in the filled-in Julia component of $P_{a,b}$ containing $z_{a,b}^i$ for all $i=1,\cdots,k$.
\end{enumerate}

Since $(a_0,b_0)$ is contained in the bifurcation locus of $P_{a,b}$ and $\sharp\,Y_{a,b}=k\geq 4$, there exist $(a,b)\in U'\setminus\{(a_0,b_0)\}$ and $N\geq 1$ such that $P_{a,b}^{\circ N}(-a)\in K_{a,b}(z_{a,b}^i)$ for some $i\in\Z/k\Z$. By the property of $z_{a,b}^i$, it means that the filled-in Julia component $K_{a,b}(-a)$ is non-periodic. According to Theorem \ref{thm:pt-or-copy}(b), the Julia set of $P_{a,b}$ is a Cantor set and $\mathscr{B}_{a,b}^+(\zeta)$ is a point.
\end{proof}

Proposition \ref{prop:accum} implies that the point components of $\mathscr{B}^+(\zeta)$ (corresponding to Cantor Julia sets) are dense in $\partial \mathscr{B}^+(\zeta)$. Actually, applying ``patterns" and ``parapatterns" in \cite{BH92}, one can show that for any non-degenerated component $\mathscr{B}_{a,b}^+(\zeta)$ and any open set $U$ in $\mathscr{L}^+(\zeta)$ with $U\cap \partial\mathscr{B}_{a,b}^+(\zeta)\neq\emptyset$, the set $U$ contains another smaller Mandelbrot copy (a component of $\partial \mathscr{B}^+(\zeta)$) which is different from $\mathscr{B}_{a,b}^+(\zeta)$.

\section{Hausdorff dimensions under the parabolic bifurcations}\label{sec-Shi}

In this section we shall use Shishikura's result on the Hausdorff dimension of the Julia sets near the parabolic parameters to prove the following result:

\begin{thm}\label{thm-dense}
The parameters such that the Julia sets of $P_{a,b}$ are Cantor sets having Hausdorff dimension two are dense in $\partial\mathscr{B}^+(\zeta)$.
\end{thm}

For a rational map, recall that the \textit{\textit{\textit{parabolic basin}}} $\MA(z_0)$ of a parabolic periodic point $z_0$ is the set of points which are iterated to $z_0$ in a neighborhood of the initial points. The \textit{immediate parabolic basin} of $z_0$ is the union of periodic connected components of $\MA(z_0)$.
The following theorem is a weak version which has been proved in \cite[Theorem 2]{Shi98}.

\begin{thm}[{Shishikura}]\label{thm-shishi}
Suppose that a rational map $f_0$ of degree $d\geq 2$ has a parabolic fixed point $z_0$ with multiplier $e^{2\pi\ii p/q}$ ($p$, $q\in\Z$, $(p,q)=1$) and that the immediate parabolic basin of $z_0$ contains only one critical point of $f_0$. Then for any $\varepsilon>0$ and $b>0$, there exist a neighborhood $\MN$ of $f_0$ in the space of rational maps of degree $d$, a neighborhood $V$ of $z_0$ in $\EC$, positive integers $N_1$ and $N_2$ such that if $f\in\MN$, and if $f$ has a fixed point in $V$ with multiplier $e^{2\pi\ii\alpha}$, where
\begin{equation}\label{equ:beta-dom}
q\alpha=p\pm\frac{1}{a_1\pm\frac{1}{a_2+\beta}}
\end{equation}
with integers $a_1\geq N_1$, $a_2\geq N_2$ and $\beta\in\C$, $0\leq\re\beta<1$, $|\im\beta|\leq b$, then
\begin{equation*}
\textup{H-dim}(J(f))>2-\varepsilon.
\end{equation*}
\end{thm}

\begin{rmk}
The original statement of Theorem \ref{thm-shishi} in \cite{Shi98} is stronger: the dimension inequality holds for the \textit{hyperbolic dimension} of $f$, which is a lower bound the Hausdorff dimension of the Julia set of $f$.
\end{rmk}

For any fixed $p,q\in\Z$ with $(p,q)=1$, we define
\begin{equation*}
\Omega_{p,q}:=\{\alpha\in\C:\alpha \text{ satisfies \eqref{equ:beta-dom}}\}.
\end{equation*}
The following lemma can be obtained by a direct calculation. See Figure \ref{Fig_Collar}.

\begin{lema}\label{lema:nbd}
The interior of $\Omega_{p,q}$ is an open set consisting of countable many open disks $(V_n)_{n\in\N}$  which satisfies $V_n\cap\R\neq\emptyset$ for all $n\in\N$, $\diam(V_n)\to 0$ and $V_n\to p/q$ as $n\to\infty$.
\end{lema}

\begin{figure}[!htpb]
  \setlength{\unitlength}{1mm}
  \centering
  \includegraphics[width=0.8\textwidth]{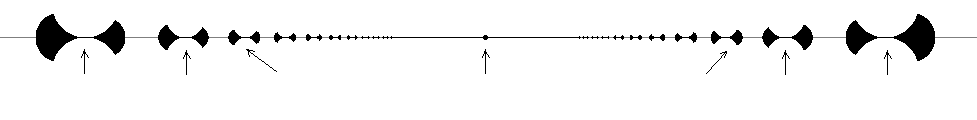}
  \put(-96,1){\small{$-\tfrac{1}{N_1}$}}
  \put(-88,1){\small{$-\tfrac{1}{N_1+1}$}}
  \put(-77,1){\small{$-\tfrac{1}{N_1+2}$}}
  \put(-11,1){\small{$\tfrac{1}{N_1}$}}
  \put(-23,1){\small{$\tfrac{1}{N_1+1}$}}
  \put(-32,1){\small{$\tfrac{1}{N_1+2}$}}
  \put(-51.7,1){\small{$0$}}
  \caption{The domain $\Omega_{p,q}$ of $\alpha$, where $p=0$ and $q=1$.}
  \label{Fig_Collar}
\end{figure}
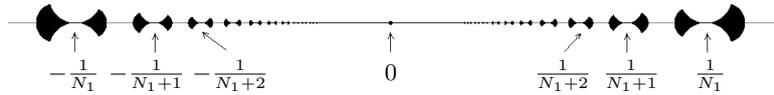

\begin{proof}[{Proof of Theorem \ref{thm-dense} and the Main Theorem}]
We fix $\zeta\in\C\setminus\overline{\D}$. By Theorem \ref{thm:pt-or-copy}, $\mathscr{B}^+(\zeta)$ consists of countably many non-degenerated components which are the copies of the Mandelbrot set. We label them by $(M^n)_{n\geq 1}$.

Let $P_{a_1,b_1}$ be a parabolic polynomial in $M^1$ such that $P_{a_1,b_1}$ has a parabolic periodic point. According to Theorem \ref{thm-shishi} and Lemma \ref{lema:nbd}, there exists an open disk $U_1$ in $\mathscr{L}^+(\zeta)$ such that $U_1\cap \partial M^1\neq\emptyset$ and the Julia sets of all $P_{a,b}$ in $U_1$ have Hausdorff dimension at least $1$.

By Proposition \ref{prop:accum}, there exists a degenerated connected component $(a_1',b_1')$ of $\mathscr{B}^+(\zeta)$ contained in $U_1$. Let $U_1'\subset U_1$ be an open neighborhood of $(a_1',b_1')$ such that $U_1'\cap M^1=\emptyset$. Note that $(a_1',b_1')$ is contained in the bifurcation locus of \eqref{equ:holo-fami}. According to \cite{McM00}, $U_1'\cap \mathscr{B}^+(\zeta)$ contains a copy of the Mandelbrot set (this copy may be a proper subset of a connected component of $\mathscr{B}^+(\zeta)$).

Next we choose any parabolic polynomial $P_{a_2,b_2}$ (which is not in $M^1$) in $U_1'$. Still by Theorem \ref{thm-shishi} and Lemma \ref{lema:nbd}, there is a disk $U_2\Subset U_1'$ in $\mathscr{L}^+(\zeta)$ such that $U_2\cap \partial \mathscr{B}^+(\zeta)\neq\emptyset$ and the Julia sets of all $P_{a,b}$ in $U_2$ have Hausdorff dimension at least $3/2$. Inductively, one can choose a sequence of parabolic parameters $(a_n,b_n)_{n\geq 1}\subset\partial\mathscr{B}^+(\zeta)$ and a sequence of nested disks $(U_n)_{n\geq 1}$ in $\mathscr{L}^+(\zeta)$ with $U_n\Subset U_{n-1}$ such that $(a_n,b_n)$ is not in $M^1\cup\cdots\cup M^{n-1}$, $U_n\cap \partial \mathscr{B}^+(\zeta)\neq\emptyset$, and the Julia sets of all $P_{a,b}$ in $U_n$ have Hausdorff dimension at least $2-1/n$.

Note that $(U_n)$ are nested open disks and $U_n\Subset U_{n-1}$ for all $n$. Hence $\cap_{n\geq 1} U_n=\cap_{n\geq 1} \overline{U}_n$ is nonempty. Therefore, we can obtain a point $(a^*,b^*)\in\cap_{n\geq 1} U_n$ such that the Julia set of $P_{a^*,b^*}$ has Hausdorff dimension two. By the choice of $(a_n,b_n)$, the point $(a^*,b^*)$ is not contained in any copies of the Mandelbrot set of $\mathscr{B}^+(\zeta)$. According to Theorem \ref{thm:pt-or-copy}, the Julia set of $P_{a^*,b^*}$ must be a Cantor set.

Let $U$ be any open set in $\mathscr{L}^+(\zeta)$ such that $U\cap \partial\mathscr{B}^+(\zeta)\neq\emptyset$.
By Proposition \ref{thm:pt-or-copy} and \cite{McM00}, the parabolic parameters are dense in $\partial\mathscr{B}^+(\zeta)$. Hence there exists a parabolic parameter $(a_1,b_1)\in U$. According to Lemma \ref{lema:nbd}, the open disk $U_1$ in the proof above can be chosen such that $U_1\subset U$. Then we have $(a^*,b^*)\in U$. This ends the proof of Theorem \ref{thm-dense} and hence the Main Theorem.
\end{proof}

\section{Some other examples}\label{sec-exam}

In this section we generalize the Main Theorem in the introduction to some higher degree polynomials and some rational maps.

\subsection{Higher degree polynomials}

For $d\geq 3$, we define
\begin{equation*}
Q_{a,b}(z)=d\int_0^z(w-a)^{d-2}(w+(d-2)a)\,\textup{d}w+b, \text{ where }a,b\in\C^*.
\end{equation*}
It is easy to check that $Q_{a,b}$ is a monic and centered polynomial with degree $d\geq 3$, which has two critical points $a$ and $-(d-2)a$ with local degrees $d-1$ and $2$ respectively. In particular, $Q_{a,b}=P_{a,b}$ if $d=3$ and $Q_{a,b}(z)=z^4-6a^2z^2+8a^3 z+b$ if $d=4$.

\begin{thm}
For any $d\geq 3$, there exist polynomials with degree $d$ whose Julia sets are Cantor sets with Hausdorff dimension two.
\end{thm}

\begin{proof}
The polynomial $Q_{a,b}$ can be studied completely similarly as the cubic polynomials $P_{a,b}$ by patterns and parapatterns. See \cite[\S12]{BH92}. In particular, Theorem \ref{thm:pt-or-copy} holds for $Q_{a,b}$. Then applying a similar argument as $P_{a,b}$ one can obtain some Cantor Julia sets of $Q_{a,b}$ with Hausdorff dimension two.
\end{proof}

In general, if one has obtained some special properties of the Julia sets of polynomials (such as full Hausdorff dimension or positive area etc.), then a natural idea to generalize the results to higher degree rational maps is to use \textit{polynomial-like mapping} theory (see \cite{DH85}). However, for Cantor Julia sets this theory cannot be used directly.

\subsection{Generalized McMullen maps}

Let us consider the following generalized McMullen family
\begin{equation*}
f_{a,b}(z)=z^n+\frac{a^2}{z^n}+b, \text{ where } n\geq 2 \text{ and } a,b\in\C^*.
\end{equation*}
This family was introduced in \cite{BDGR08} and has been studied thoroughly in \cite{XQY14}. The critical points of $f_{a,b}$ are $c_k^\pm:=\pm\sqrt[n]{a}\,e^{2\pi\ii k/n}$ $(k=1,\cdots,n)$ and the critical values of $f_{a,b}$ are $v_\pm=b\pm 2a$. Hence $f_{a,b}$ has exactly two independent critical orbits essentially.

Note that $f_{a,b}$ has a super-attracting fixed point at $\infty$, $f_{a,b}^{-1}(\infty)=\{0,\infty\}$ and the local degrees of $f_{a,b}$ at $0$ and $\infty$ are both $n$. In order to construct Cantor Julia sets, we assume that the immediate super-attracting basin of $\infty$ contains a critical value $v^+$ or $v^-$. According to B\"{o}ttcher's theorem, there exists a unique conformal isomorphism $\varphi_{a,b}$ defined in a neighborhood of $\infty$ which is tangent to the identity at $\infty$ such that $\varphi_{a,b}$ conjugates $f_{a,b}$ to $\zeta\mapsto\zeta^n$. Moreover, $\varphi_{a,b}$ can be extended to a conformal isomorphism
\begin{equation*}
\varphi_{a,b}:V_{a,b}\to \C\setminus\overline{\D}_r,
\end{equation*}
where $r>1$ such that $\partial V_{a,b}$ contains the critical points $\{c_k^+:$ $k=1$, $\cdots$, $n\}$ or $\{c_k^-:$ $k=1$, $\cdots$, $n\}$. The complement $\C\setminus\overline{V}_{a,b}$ is a Jordan disk containing $0$ and satisfying $f_{a,b}(\C\setminus\overline{V}_{a,b})=f_{a,b}(V_{a,b})\cup\{\infty\}$. See Figure \ref{Fig_XQY-J}.

\begin{figure}[!htpb]
  \setlength{\unitlength}{1mm}
  \centering
  \includegraphics[width=0.45\textwidth]{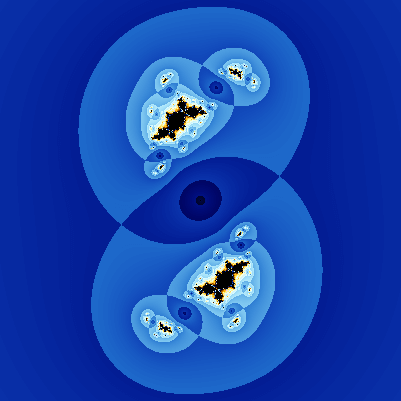}\hskip0.3cm
  \includegraphics[width=0.45\textwidth]{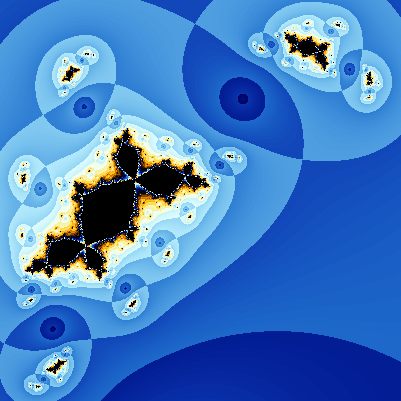}
  \caption{The Julia set of $f_{a,b}$, where $n=2$ and the parameter $(a,b)$ is chosen such that $f_{a,b}$ has exactly one escaped critical value.}
  \label{Fig_XQY-J}
\end{figure}

The potential function
\begin{equation*}
h_{a,b}(z)=\lim_{k\to\infty}\frac{1}{n^k}\log_+|f_{a,b}^{\circ k}(z)|
\end{equation*}
is defined\footnote{Actually $h_{a,b}$ can be defined in $\C\setminus\cup_{n\in\N}f_{a,b}^{-n}(\infty)$. In order to obtain the similar structures of the dynamical planes and parameter spaces as cubic polynomials, it will be convenient to ignore the definition of the potential on the preimages of $\C\setminus\overline{V}_{a,b}$.} in $\C\setminus\cup_{n\in\N}f_{a,b}^{-n}(\C\setminus\overline{V}_{a,b})$. Without loss of generality, we assume that $v_+\in f_{a,b}(\partial V_{a,b})$ and $h_{a,b}(v_-)<h_{a,b}(v_+)$.

\begin{thm}
If the Julia set of $f_{a,b}$ is a Cantor set, then the area of this Julia set is zero.
Moreover, there exists $(a,b)\in\C^*\times\C^*$ such that the Julia set of $f_{a,b}$ is a Cantor set having Hausdorff dimension two.
\end{thm}

\begin{proof}
We only give a sketch of the proof here since the idea is parallel to that of cubic polynomials in \cite{BH92}. The definitions of $h_{a,b}$ and $\varphi_{a,b}$ above for $f_{a,b}$ are sufficient for us to define Branner-Hubbard puzzles as cubic polynomials. In particular, patterns and parapatterns can be defined similarly. Unlike the definition of $\mathscr{L}^+(\zeta)$ for cubic polynomials, for $f_{a,b}$ we define
\begin{equation*}
\mathscr{L}^+(\zeta):=\{(a,b)\in\C^2: \varphi_{a,b}(v_+)=\zeta \text{ and }h_{a,b}(v_-)<h_{a,b}(v_+)\}.
\end{equation*}
One can show that $\mathscr{L}^+(\zeta)$ is a Riemann surface with infinitely many boundary components.
Similar to the case of cubic polynomials, we define
\begin{equation*}
\mathscr{B}^+(\zeta):=\{(a,b)\in\mathscr{L}^+(\zeta): h_{a,b}(v_-)=0\}.
\end{equation*}

\begin{figure}[!htpb]
  \setlength{\unitlength}{1mm}
  \centering
  \includegraphics[width=0.45\textwidth]{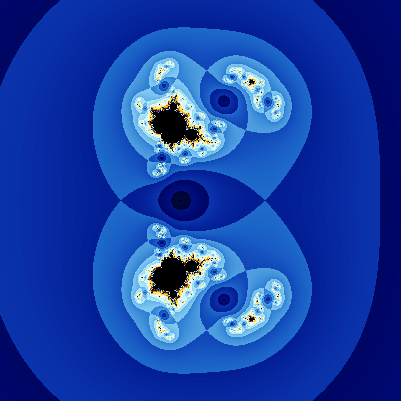}\hskip0.3cm
  \includegraphics[width=0.45\textwidth]{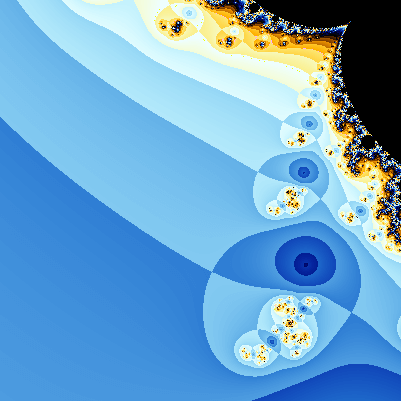}
  \caption{The space $\mathscr{L}^+(\zeta)$ and its zoom for some $\zeta>1$, where $n=2$. The set $\mathscr{B}^+(\zeta)\subset\mathscr{L}^+(\zeta)$ consists of countably many copies of the Mandelbrot set and uncountably many points.}
  \label{Fig_XQY-M}
\end{figure}

For the dynamical planes, a similar statement to Theorem \ref{thm:pt-or-copy} for $f_{a,b}$ has been proved in \cite[Theorem 4.1]{XQY14}.
Applying the relation between parameter spaces and dynamical planes as in \cite{BH92}, one can obtain that each component of $\mathscr{B}^+(\zeta)$ is either a point or a copy of the Mandelbrot set (see Figure \ref{Fig_XQY-M}). The statement on the zero area can be proved similarly as \cite[Theorems 5.9 and 12.6]{BH92}. The existence of Cantor Julia sets of $f_{a,b}$ with full dimension can be proved similarly as the Main Theorem.
\end{proof}


\end{document}